\documentclass[12pt]{amsart}
\usepackage{amsmath,amsfonts}
\usepackage{amsthm}
\usepackage{amssymb}
\usepackage{fancyhdr}
\usepackage{stmaryrd}
\usepackage{graphicx}
\usepackage{setspace}
\usepackage{verbatim}\usepackage{scalerel,amssymb}

\usepackage[a4paper,left=3cm,right=3cm,top=3cm,bottom=4.5cm]{geometry}
\usepackage[hypcap]{caption}
\usepackage[utf8]{inputenc}
\usepackage{amsmath,amsthm}
\usepackage{comment}
\usepackage{etaremune}
\usepackage{amssymb}
\usepackage{stmaryrd}
\usepackage{graphicx}
\usepackage{hyperref}
\usepackage[all,2cell]{xy}\UseAllTwocells \SilentMatrices
\usepackage{color}
\newdir{ >}{{}*!/-10pt/@{>}}
\usepackage{pdfsync}
\usepackage{tikz-cd}
\theoremstyle{definition}
	\newtheorem{thm}{Theorem}[section]
	\newtheorem*{thm*}{Theorem}	\newtheorem{Def}[thm]{Definition}
	\newtheorem*{Def*}{Definition}
	\newtheorem{rmk}[thm]{Remark}
	\newtheorem*{rmk*}{Remark}
	\newtheorem{cor}[thm]{Corollary}
	\newtheorem*{cor*}{Corollary}
	\newtheorem{prop}[thm]{Proposition}
	\newtheorem*{prop*}{Proposition}
	\newtheorem{eg}[thm]{Example}
	\newtheorem*{eg*}{Example}
	\newtheorem{lem}[thm]{Lemma}
	\newtheorem*{lem*}{Lemma}
	
	\newtheorem*{ex*}{Exercise}
	
	\newtheorem*{claim*}{Claim}
	\newtheorem{fact}[thm]{Fact}
\theoremstyle{definition}
	
	\newtheorem{soln*}{Solution}
	
	\newtheorem*{note*}{Note}
	
	\newtheorem{question}[thm]{Question}

\newcommand{\conv}{\mathrm{conv}}

\newcommand{\st}{\mathrm{st}}

\newcommand{\RES}{\mathbf{RES}}
\newcommand{\VF}{\mathbf{VF}}
\newcommand{\VG}{\Gamma}
\newcommand{\Stab}{\text{Stab}}
\newcommand{\res}{\text{res}}
\newcommand{\val}{\mathrm{val}}

\usepackage{lipsum}

\title{A note on $\mu$-stabilizers in ACVF}
\author{Jinhe Ye}
\address{Institut de math\'ematiques de Jussieu – Paris Rive Gauche}
\email{jinhe.ye@imj-prg.fr}
\urladdr{https://sites.google.com/view/vincentye}
\date{}
\begin{document}
\maketitle
\begin{abstract}
 We study $\mu$-stabilizers for groups definable in ACVF in the valued field sort. We prove that $\Stab^\mu(p)$ is an infinite unbounded definable subgroup of $G$ when $p$ is standard and unbounded. In the particular case when $G$ is linear algebraic, we show that $\Stab^\mu(p)$ is a solvable algebraic subgroup of $G$, with $\dim(\Stab^\mu(p))=\dim(p)$ when $p$ is $\mu$-reduced and unbounded.
\end{abstract}
\author
\section{Introduction}
In \cite{Kobi}, working in an o-minimal theory, given a definable group $G$ and an unbounded definable curve $C\subseteq G$, Peterzil and Steinhorn canonically associated to $C$ a one-dimensional definable subgroup $H_C \subseteq G$. Intuitively, this group is the ``limit of the tangent space" towards infinity along the curve. In fact, such groups depend only on the type $p$ of a branch of the curve at infinity.  Subsequently, in \cite{Ser}, Peterzil and Starchenko developed a general machinery to study definable topological group actions, with an emphasis on groups definable in o-minimal theories. $H_C$ above is exactly $\mu$-stabilizer of the type $p$ of a branch of the curve at infinity in this language. The group action $G$ on its space of $\mu$-types is a model-theoretic analogue of the Samuel compactification of $G$ (See ~\cite{sam}). Following their approach, several generalizations have been made to the purely algebraic setting~\cite{muACF} and the $p$-adic setting~\cite{WY_padic}. In our recent work \cite{muACF} with Kamensky and Starchenko, we generalized the notion of $\mu$-stabilizers to the context of linear algebraic groups over trivially valued algebraically closed fields and analyzed the dimension and structure of such groups, despite the fact that the Zariski topology is too coarse to give a meaningful notion of the ``infinitesimal neighborhood" around the identity. Nonetheless, the $\mu$-stabilizers in this setting is canonical in the following sense: Let $X$ be a $G$-variety with a $G$-equivariant embedding of $G$ into $X$, and $p$ be a $G$-type. If $X$ contains a realization $c$ of $p$, then $\Stab^\mu(p)$ is contained in the stabilizer of $c$ under the $G$-action on $X$. See~\cite[Remark 3.4]{muACF} for details. Particularly, we prove the following, the precise meaning of the terms will appear in Section~\ref{sec:lin-alg}:
\begin{fact}[{\cite[Theorem 1.2]{muACF}}]\label{fact:ksy}
  Let \(G\) be a linear algebraic group defined over an algebraically closed field \(k\), considered as a trivially valued field embedded in an algebraically closed valued field. Let \(p\in{}S_G(k)\) be an ACVF type that is residually algebraic and centered at infinity, then \(\Stab^\mu(p)\) is 
  infinite.
  
  Furthermore, if \(p\) is \(\mu\)-reduced, let \(V\) denote the Zariski 
  closure of \(p\) over \(k\). Then \(\dim(\Stab^\mu(p))=\dim V=\dim{}p\).  
  Moreover, for each type \(p\in{}S_G(k)\), \(\Stab^\mu(p)\) is a solvable 
  linear algebraic group.
\end{fact}
Further more, as illustrated in~\cite{muACF}, the space of $\mu$-types are naturally associated with the Zariski-Riemann space of $G$ and is a universal compactification of $G$ in an appropriate sense. We hope the study of $\mu$-stabilizers can help us understand the Zariski-Riemann space and definable compactifications of $G$.

It is therefore natural to ask if such a machinery can be developed in the context of groups $G$ defined over an algebraically closed valued field $F$, equipped with the valuation topology. The intersection of all the $F$-definable (in the 3-sorted language of ACVF) neighborhoods around the group identity gives rise to a natural ``infinitesimal neighborhood". Hence the setting resembles more of \cite{Ser} than \cite{muACF}. However, the techniques in \cite{muACF} still apply. In particular, the correspondence between definably connected components and types in the orbit of a $\mu$-reduced type remains present. From this, we obtain the main results of the paper. Before stating the results, let us briefly describe the meaning of the terminologies involved. ``Unbounded" means that the type does not concentrate on a bounded ball (see Definition~\ref{def:unbounded}), an analogue of unbounded in the o-minimal counterpart, or centered at infinity in the valuation-theoretic terminology. And ``standard" (see Definition~\ref{def:standard}) is the analogue of residually algebraic in~\cite{muACF}.

The above connections are made precise in Proposition~\ref{prop:res-alg}. In contrast to stably dominated types that have been studied extensively in~\cite{HHM,HL}, standard types are controlled by $\Gamma$. Guided by the intuition that algebraically closed valued field (ACVF) is governed by the residue field (stable) sort and the $\Gamma$-sort, we believe that understanding standard types together with the stably dominated types would shed some light on all definable types in ACVF. Lastly, ``$v$+$g$-continuous" (see Definition~\ref{def:v+g-cont}) is the same as continuous when viewed as functions on the model-theoretic Berkovich analytification using stably dominated types~\cite{HL}.

\begin{thm}\label{thm:main1}
Let $G$ be a definable group in $\VF^n$ over an algebraically closed valued field $F$ considered as constants in the $\VF$-sort. Assume further that $G$ is closed with respect to the valuation topology and the group operations on $G$ are continuous. For a unbounded standard $G$-type $p$ over $F$, $\Stab^\mu(p)$ is a definable subgroup of $G$. Moreover, if $G$ is $g$-closed and the group operations are $g$-continuous, $\Stab^\mu(p)$ is unbounded and infinite.
\end{thm}
\begin{rmk}\label{rmk:lin-alg}
For a general definable group $G$, it is hard to study the dimension of the stabilizers. However, when $G$ is linear algebraic, we have that $\Stab^\mu(p)\subseteq G$ is a solvable algebraic subgroup, and $\dim(\Stab^\mu(p))=\dim(p)$ for $\mu$-reduced $p$ by reducing to \cite[Theorem 1.2]{muACF}. See Section~\ref{sec:lin-alg} for details.

\end{rmk}
The paper is structured as the following: In Section~\ref{sec:mu}, we set up the language of $\mu$-types and their stabilizers in ACVF. In Section~\ref{sec:standard-part}, we relate the $\mu$-stabilizers for standard types with an actual ``standard" part map (See Definition~\ref{def:standard}). In Section~\ref{sec:tame-top}, we gather the ingredients in~\cite{HL} needed. Section~\ref{sec:proof} is dedicated to the proof of the main results with an adaptation of the techniques in~\cite{muACF}.
\paragraph{\textbf{Acknowledgements.}}
We would like to thank Antoine Ducros, Fran\c cois Loeser, Kobi Peterzil and Sergei Starchenko for their helpful conversations. The author was partially supported by NSF research grant DMS1500671 and Fran\c cois Loeser's Institut Universitaire de France grant. We thank the anonymous referee for numerous comments that helped improving the paper greatly.
\section{\texorpdfstring{$\mu$}{mu}-types in ACVF}
\subsection{Definable group actions in ACVF}\label{sec:mu}
In this section, we develop the language of $\mu$-types and their stabilizers in an algebraically closed valued field. Most of the facts stated below are results in \cite{Ser} or \cite{muACF} adapted to our setting. Since we are working in a model of ACVF, essentially all the definitions are from \cite{Ser} translated into the valuative context. We include them for the sake of completeness. \textbf{From this section on, the underlying theory is fixed to be a completion of the theory of algebraically closed valued fields (ACVF).} For general facts concerning ACVF, we refer the readers to \cite{Lou}. We work in $\mathcal{L}_\val$, the usual 3-sorted language of ACVF. We use $\VF$ to denote the valued field sort and $\Gamma$ to denote the sort for the value group without the point $\infty=v(0)$. 
\begin{Def}
Let $G$ be a group defined in the $\VF$-sorts over some $F\models \mathrm{ACVF}$. By a \textbf{$G$-formula} $\phi(x)$, we mean a formula $\phi(x)$ such that $\phi(x)\models x\in G$. For a (small) set of parameters $A$, we use $\mathcal{L}_G(A)$ to denote the set of $G$-formulas over $A$. Lastly, by 
\textbf{a (partial) $G$-type}, we mean a small set of consistent $G$-formulas. We use $S_G(F)$ to denote the set of complete $G$-types over $F$.
\end{Def}
\begin{Def}
For $G$ a definable group as above, let $\varphi(x),\psi(x)$ be $G$-formulas, we use $\varphi\cdot \psi(x)$ to denote the formula:
\begin{align*}\exists x_1\exists x_2( \varphi(x_1)\wedge\psi(x_2)\wedge x=x_1\cdot x_2).
\end{align*}
Similarly for partial $G$-types $\Sigma(x),\Pi(x)$,
$\Sigma \cdot \Pi(x)$ is the following partial type:
\begin{align*}
    \{\varphi\cdot \psi(x):\varphi \in \Sigma,\, \psi \in \Pi\}.
\end{align*}
\end{Def}
Note that $G(F)$ acts on the $G$-formulas and $G$-types as well via the following: $g\cdot \varphi(x)$ is defined to be $(x=g)\cdot \varphi$. The set it defines is the same as the set obtained by multiplying $g$ by the set defined by $\varphi(x)$. The action on $G$-types is defined similarly.

\begin{Def}
Let $\Sigma(x)$ be a partial $G$-type over a small set $A$. We define $\Stab(\Sigma)$ to be 
\begin{align*}
    \{g \in G(F): \text{ for all } \phi \in \mathcal{L}_G(A), \Sigma \models \phi \text{ iff } \Sigma\models g\cdot \phi \}.
\end{align*}

For each $\phi \in \mathcal{L}_G(A)$, let $B_{\Sigma,\phi}$ be the set 
\[
\{g \in G(F): \Sigma(x) \models g\cdot \phi\}.
\]
We define $\Stab_\phi(\Sigma)=\Stab(B_{\Sigma,\phi})=\{g\in G(F): g\cdot B_{\Sigma,\phi}=B_{\Sigma,\phi}\}$, it follows immediately that it is a subgroup of $G(F)$.
\end{Def}
\begin{rmk}
Note that we can similarly develop the notion of right stabilizers by considering $G$ acting on the right. The construction in the paper later on can be modified accordingly for the right action. However,  the left and right stabilizers of the same type may not agree in general, see Example~\ref{ex:basic1} and~\ref{ex:basic2}.
\end{rmk}
The following is immediate.
\begin{fact}
$\Stab(\Sigma)=\bigcap_{\phi \in \mathcal{L}_G(A)} \Stab_\phi(\Sigma)= \bigcap_{\Sigma\models \phi}\Stab_\phi(\Sigma)$.
\end{fact}
We say a (partial) type $\Sigma(x)$ over $F$ is \textbf{definable} if for each $\phi(x;y)$, there is $\psi(y)\in \mathcal{L}_\val(F)$ such that $ \Sigma\models \phi(x;c)$ iff $ F\models \psi(c)$.
We have the following.
\begin{fact}[{\cite[Proposition 2.13]{Ser}}]\label{fact:Stab-def}
Let $\Sigma(x)$ be a partial definable $G$-type over $F$ of ACVF, $\Stab(\Sigma)$ is type-definable over $F$.
\end{fact}
\textbf{From now on, $F$ is a given algebraically closed valued field embedded in the $\VF$-sort, and $G$ is a definable group over $F$ in the valued field sort and $e$ is the group identity. We assume throughout that for some $n\in\mathbb{N}$, $G\subseteq \VF^n$ is closed in the valuation topology and the group operations are continuous with respect to the valuation topology.} $G(F)$ acts on the space $S_G(F)$ by left multiplication. As in \cite{Ser}, there is a canonical $F$-type-definable subgroup of $G$, denoted by $\mu$, given by the intersection of all the $F$-definable (valuative) neighborhoods of $e$ in $G$, where $e$ denotes the group identity. It follows immediately that for any $g\in G(F)$, $g\cdot \mu\cdot g^{-1}=\mu$.

For each tuple of $\VF$-variables $x=(x_1,...,x_n)$, we use $B_{>\gamma,x}$ to denote the $n$-dimensional ball/polydisc $\{y:v(y_i-x_i)>\gamma\text{ for } i=1,...,n\}$, and $B_{\geq\gamma,x}$ is defined similarly. For $g\in G$, $G_{>\gamma,g}$ is used to denote $B_{>\gamma,g}\cap G$ and similarly for $G_{\geq\gamma,g}$. We use $\nu$ and $\nu^n$ to denote the intersection of all $F$-definable neighborhoods of the origin in $\VF$ and $\VF^n$ respectively. In this notation, $\mu=(e+\nu^n)\cap G$.

The group $\mu$ induces the following equivalence relation on $S_G(F)$: $p\sim_\mu q$ if there is $a\models p$, $b\models q$ such that there is $\epsilon \models \mu$ with $a=\epsilon\cdot b$. In other words, $p\sim_\mu q$ iff $\mu\cdot p=\mu\cdot q$. We use $p_\mu$ to denote the equivalence class of $p$, and refer to it as the \textbf{$\mu$-type} of $p$. $S^\mu_G(F)$ denotes the quotient of $S_G(F)$ under the above equivalence relation. The $G(F)$-action on $S_G(F)$ respects the equivalence relation. Hence we have an induced $G(F)$-action on $S^\mu_G(F)$.
For $p \in S_G(F)$, we use $\Stab^\mu(p)$ to denote $\Stab(\mu\cdot p)$. Clearly, it is also the stabilizer of $p_\mu$ under the induced $G(F)$-action on $S^\mu_G(F)$. More precisely, $\Stab^\mu(p)$ consists of the set of elements $g\in G(F)$ such that for any $a\models p$, there is $\epsilon \models \mu$ and $b\models p$ with $g\cdot a=\epsilon\cdot b$.

\begin{fact}[{\cite[Claim 2.15]{Ser}}]\label{fact:def}
Let $p \in S_G(F)$ be definable over $F$. Then $\mu\cdot p$ is definable over $F$ as a partial type.
\end{fact}

The above facts show that $\Stab^\mu(p)$ is a type-definable subgroup by~\cite[Proposition 2.13]{Ser}. Unlike the o-minimal or algebraic counterpart, ACVF does not have descending chain condition for definable groups. Hence, a priori, it is not clear if $\Stab^\mu(p)$ is definable, even for definable types $p$.

We recall one last fact from \cite{Ser}.
\begin{fact}[{\cite[Claim 2.18]{Ser}}]\label{fact:conj}
Let $g \in G(F)$ and $p,q\in S_G(F)$ be such that $g\cdot p_\mu=q_\mu$, then $\Stab^\mu(q)=g\Stab^\mu(p)g^{-1}$.
\end{fact}
Lastly, we introduce the following terminology.
\begin{Def}
We say $p\in S_G(F)$ is \textbf{$\mu$-close} to a definable set $X$ if $\models\mu\cdot p\cap X\neq \emptyset$.
\end{Def}

\subsection{Computing \texorpdfstring{$\mu$}{mu}-stabilizers via standard parts}\label{sec:standard-part}
In this section, we will compute the $\mu$-stabilizers of $p$ using the standard part map.

Working in $\mathbb{U}$, a monster model of ACVF extending $F$, we have the following statement analogous to~\cite[Claim 2.22]{Ser}.
\begin{prop}\label{prop:st}
Let $p \in S_G(F)$, then for any $a \models p$,
\begin{equation*}
     \Stab^\mu(p)=(\mu\cdot p(\mathbb{U}))\cdot (\mu\cdot p(\mathbb{U}))^{-1} \cap F^n=\mu \cdot p(\mathbb{U})\cdot a^{-1}\cap F^n.
\end{equation*}
\end{prop}
\begin{proof}
Let us begin by proving the second equality 
\begin{equation*}
    (\mu\cdot p(\mathbb{U}))\cdot (\mu\cdot p(\mathbb{U}))^{-1} \cap F^n=\mu \cdot p(\mathbb{U})\cdot a^{-1}\cap F^n.
\end{equation*} 
Clearly $\mathbf{RHS}\subseteq \mathbf{LHS}$ by definition. For the reverse containment, let $h=\epsilon_1\cdot c\cdot b^{-1}\cdot \epsilon_2 \in \mathbf{LHS}$ for some $b,c \in p(\mathbb{U})$. Since $h \in G(F)$ and $g\cdot \mu\cdot g^{-1}=\mu$ for any $g \in G(F)$, so $h\cdot{\epsilon_2}^{-1}=\epsilon_3\cdot h$. Replacing $\epsilon_1$ by $\epsilon_3^{-1}\cdot \epsilon_1$ we may assume that $\epsilon_2=e$, the group identity. Take an automorphism $\sigma$ of $\mathbb{U}$ over $F$ such that $\sigma(b)=a$, we have $h=\sigma(h)=\sigma(\epsilon_1)\cdot \sigma(c)\cdot a^{-1}$.

Next we show $\Stab^\mu(p)=\mu \cdot p(\mathbb{U})\cdot a^{-1}\cap F^n$. Clearly, if $g \in \Stab^\mu(p)$, we have that $g\cdot a \models \mu\cdot p$. Conversely, if $g\in G(F) \text{ and } g\cdot a\in \mu\cdot p(\mathbb{U})$, then for any $b \models p$ $g\cdot b \in\mu\cdot p(\mathbb{U})$. Hence $g \in \Stab^\mu(p)$.
\end{proof}
However, the above description is not so helpful in practice. To show our main theorem, we wish to realize $\Stab^\mu(p)$ via some definable objects as in~\cite[Section 3.2]{Ser}. In \cite{muACF}, the class of residually algebraic types was isolated for this purpose and their $\mu$-stabilizers can be described via a standard part map. This motivates us to define the following analogue of residually algebraic types in ACVF. Namely, they are the types for which a canonical ``standard part" map to $F$ can be defined. We changed our terminology from ``residually algebraic" to avoid potential ambiguity, since $F$ is equipped with a well-defined residue field. 

Let $K$ be an algebraically closed field (properly) extending $F$, satisfying the following property: 
\begin{equation}\label{eqn:standard}
\forall b \in K ((\exists a \in F\,\, v(b)>v(a))\Rightarrow (\exists a \in F\, \forall  c \in F^\times \,\, v(b-a)>v(c))).
\end{equation}
 If $b \in K$ satisfies $v(b)>v(a)$ for some $a \in F$, ~(\ref{eqn:standard}) implies that there is a unique $b' \in F$ such that $v(b-b')>\Gamma(F)$. We use the $\st$ to denote the above map $b\mapsto b'$. We use $\conv_F(x)$ to denote the following $F$-$\bigvee$-definable set,
\[
\bigcup_{c\in F} \{x: v(x)>v(c)\}.
\]
For any $K$ satisfying~(\ref{eqn:standard}), $\conv_F(K)$ is a valuation subring of $K$ with maximal ideal $\nu(K)$ defined by
\[
\nu(K)=\{x\in K: v(x)>\Gamma(F)\}.
\]
Moreover, on $\conv_F(K)$, $\st$ is exactly the residue map of $\conv_F(K)$ and $F$ is an embedded residue field. Abusing notation, the map $(\conv_F(K))^n\to F^n$, $(a_1,...,a_n)\mapsto (\st(a_1),...,\st(a_n))$ will also be denoted by $\st$.

\begin{Def}\label{def:standard}
We will call extensions $F\preceq K$ satisfying~(\ref{eqn:standard}) \textbf{standard}. We say a type $p\in S_n(F)$ is \textbf{standard} if it is realized in a standard extension $K$ of $F$. 
\end{Def}
Note that standard extensions of algebraically closed valued fields always exist. Given $F$ a model, one can take any element $b$ in the monster model such that $\infty\neq v(b)>\Gamma(F)$ and consider $K=\mathrm{acl}(F(b))$. 
\begin{rmk}\label{rmk:standard}
The name standard comes from the similarity of $\st$ and the standard part map in relatively Dedekind complete pairs of o-minimal structures~\cite{vdd-lewen-1,Pillay_def}. By \cite[Lemma 5.11]{pair-pro}, the condition $K/F$ is standard implies that $K/F$ is separated or vs-defectless in other terminologies. Recall that $K/F$ is separated or vs-defectless means the following: For any $V\subseteq K$ some $n$-dimensional vector space over $F$, there is a basis $c_1,\ldots,c_n\in V$ such that $v(\sum_i c_ia_i)=\min_i\{v(c_ia_i)\}$ for every $a_i\in F$. It also follows immediately from~(\ref{eqn:standard}) that $\Gamma(F)$ is a convex subgroup of $\Gamma(K)$ and $\res(K)=\res(F)$. Hence by \cite[Theorem 1,9]{delon}, for any tuple $a \in K$, $tp(a/F)$ is definable over $F$. 
\end{rmk}
\begin{rmk}\label{rmk:bp}
By the previous remark, a standard extension $F\preceq K$ is separated, $\Gamma(F)$ is a convex subgroup of $\Gamma(K)$ and $\res(K)=\res(F)$.
Thus they correspond exactly to the completion of separated pairs of ACVF in~\cite[7.4.2(7)]{BP}. Note that $\st$ is clearly definable in the language of pairs of ACVF in~\cite{BP}, where separated pairs of ACVF are considered with a predicate for $F$ in $K$. By~\cite[Theorem A and C]{BP}, $F$ is stably embedded as a pure ACVF. In particular, the following holds: 
Let $X\subseteq \VF^n$ be a $K$-definable subset, $\st(X(K)\cap (\conv_F(K))^n)$ is a definable subset of $F$ (in ACVF).
\end{rmk}
As in~\cite[Section 3.2]{Ser}, we have the following:
\begin{lem}\label{lem:nbhd}
Let $G$ be a $F$-definable group in the valued field sort and $F\preceq K\models \mathrm{ACVF}$ a standard extension. Then $\st(G(\conv_F(K))= G(F)$. We also have that for $g\in G(F)$, $\mu(K)\cdot g=(g+\nu^n(K))\cap G$. 
\end{lem}
\begin{proof}
We first show that given $g \in G(\conv_F(K))$, $\st(g)\in G(F)$. Assume to the contrary that $\st(g)\notin G(F)$. Since $G$ is closed in the valuation topology, there is some $\gamma\in \Gamma(F)$ such that $F\models B_{>\gamma,\st(g)} \cap G=\emptyset$. However, this implies that $g \not\in G(K)$, a contradiction. The latter part just follows from the continuity of the group operations.
\end{proof}
Now let us come back to the $\mu$-types and ``dimensions", where by dimension we mean the topological dimension in the $\VF$-sorts in ACVF. For a type $p$ in the $\VF$-sort, $\dim(p)$ is defined to be the minimum of the dimensions of the definable sets that $p$ concentrates on. It agrees with the smallest dimension of varieties $V$ over $F$ such that $p\models x\in V$. We call the irreducible $F$-variety $V$ with the smallest dimension such that $p$ concentrates on (i.e. $p\models x\in V$) the \textbf{Zariski closure} of $p$. See~\cite[Section 4.1]{Yimu} for more details.
\begin{Def}
 A $G$-type $p\in S_G(F)$ is \textbf{$\mu$-reduced} if $p$ is of minimal dimension in $p_\mu$. We say $b\in G(\mathbb{U})$ is $\mu$-reduced if $tp(b/F)$ is $\mu$-reduced.
\end{Def}
The following is immediate from the definition.
\begin{lem}
If $p$ is $\mu$-reduced and $g\in G(F)$, then $g\cdot p$ is also $\mu$-reduced.
\end{lem}
The following shows that for standard $p$, there is a $\mu$-reduced standard $q \in p_\mu$, an analogue statement of~\cite[Corollary 3.5]{Ser} or~\cite[Lemma 4.3]{muACF}.
\begin{lem}
For $p$ a standard $G$-type over $F$, there is a standard type $q\in p_\mu$ such that $q$ is $\mu$-reduced.
\end{lem}
\begin{proof}
For a given standard $p$, if $p$ is not $\mu$-reduced, there is a $F$-definable set $C$ of smaller dimension, such that $p$ is $\mu$-close to $C$. Particularly, for each $\gamma \in \Gamma(F)$, recall that $G_{>\gamma,e}=B_{>\gamma,e}\cap G$ and we have
\begin{align*}
p \models \exists y\exists z \,\,x=y\cdot z \wedge y \in G_{>\gamma,e} \wedge z \in C.
\end{align*}
Take $a\models p$ a realization in $K$, a standard extension of $F$. By the above, we have some $g^*\in G_{>\gamma,e}(K)$ and $b \in C(K)$ with $a=g^*\cdot b$. The condition $g\in G_{>\gamma,e}$ implies that $g^*$ has coordinates in $\conv_F(K)$. Let $g=\st(g^*)$. By Lemma \ref{lem:nbhd}, there is some $\epsilon \in \mu(K)$ such that $a=g\cdot \epsilon\cdot b$. Let $q'(x)=tp(b/F)$, it is standard and has smaller dimension. Moreover, $q(x)=g\cdot q'(x)$, and $q$ has the same dimension as $q'$ and $p\sim_\mu q$. Proceed until $q$ is of minimal dimension. 
\end{proof}

\subsection{Tame topology on definable sets}\label{sec:tame-top}
In this section, we recall some results from~\cite{HL}. They are presented slightly differently from the original statements to fit our usage. More specifically, we restrict our attention to definable sets and their $v$+$g$-topologies.
 \begin{Def}\label{def:v+g}
 Let $V$ be an algebraic variety over a valued field $F$,  $X\subseteq V$ is \textbf{$v$-open} if it is open for the valuation topology. $X \subseteq V$ is \textbf{$g$-open} if it is a positive Boolean combination of Zariski open/closed sets and sets of the form 
 \begin{align*}
     \{x \in U:v\circ f(x)>v\circ g(x)\},
 \end{align*}
 where $f$ and $g$ are regular functions defined on $U$, a Zariski open subset of $V$. If $Z\subseteq V$ is a definable subset of $V$, a subset $X$ of $Z$ is said to be $v$-open (respectively $g$-open) in $Z$ if $X= Z\cap Y$, where $Y$ is $v$-open (respectively $g$-open) in $V$. The complement of a $v$-open (respectively $g$-open) is called $v$-closed (respectively $g$-closed). We say $X$ is \textbf{$v$+$g$-open} (respectively \textbf{$v$+$g$-closed}) if it is both $v$-open and $g$-open (respectively both $v$-closed and $g$-closed).
 \end{Def}
The topology generated by $v$+$g$-opens will be exactly the valuation topology since the open balls are $v$+$g$-open. Note that the $v$+$g$-opens do not form a definable topology in the usual sense, as the $g$-opens are not closed under definable unions. However, the $v$+$g$-opens are closed under finite unions.

 \begin{eg}
Any ball (closed or open) is  both $v$-open and $v$-closed. However, for $\mathcal{O}=\{x:v(x)\geq 0\}$ the valuation ring, it is only $g$-closed but not $g$-open. Similarly, the maximal ideal $\mathfrak{m}=\{x: v(x)>0\}$ is $g$-open but not $g$-closed. 
 \end{eg}
 \begin{rmk}\label{rmk:vg}
 The $v$-topology is too poor to support any meaningful topological invariants. Historically, Berkovich approached non-Archimedean analytic geometry via Berkovich analytic spaces and analytifications~\cite{Ber}. Such analytic spaces carry nice topologies (locally compact, locally connected,  etc.). In~\cite{HL}, a model-theoretic Berkovich analytification $\widehat{V}$ was introduced and the topology of $\widehat{V}$ exhibits topological tameness as o-minimal topologies. Set-theoretically, for a definable set $X$, $\widehat{X}$ consists of the stably dominated/generically stable types concentrating on $X$. 
The $v$+$g$-opens are the ``traces" of the opens on $\widehat{V}$ restricted to definable sets in the following sense: For any definable set $X\subseteq V$, $\widehat{X}\subseteq \widehat{V}$ is open/closed if and only if $X$ is $v$+$g$-open/closed in $V$ respectively~\cite[Proposition 4.2.21]{HL}.
 \end{rmk}
  When restricted to definable versions of topological notions such as definably connected components, the behavior of definable sets remains relatively tame under the $v$+$g$-topology as indicated by the previous remark.
 \begin{Def}
 Let $X$ be a definable subset of $V$, an algebraic variety. We say that $X$ is \textbf{definably connected} if $X$ cannot be written as a disjoint union of two proper $v$+$g$-open subsets of $X$. We say that $X$ has \textbf{finitely many definably connected components} if $X$ can be written as a finite disjoint union of $v$+$g$-clopen in $X$, definably connected subsets. And each $v$+$g$-clopen subset above is called a \textbf{definably connected component} of $X$. 
 \end{Def}
  In~\cite{HL}, definably connectedness was defined via the topology on $\widehat{X}$, the stable completion of $X$. The discussion immediately precedes  ~\cite[Lemma 10.4.1]{HL} explained the equivalence of their definition and the one above, together with the fact that the definably connected component of $\widehat{X}$ is of the form $\widehat{U}$ for some $U$ $v$+$g$-clopen in $X$.
 \begin{eg}
      Let $B_1$, $B_2$ be two disjoint balls (open or closed), it is easy to see that $B_1\cup B_2$ is not definably connected. On the other hand, any $B\subseteq \VF$ ball (open or closed) is definably connected. This follows from the description of the standard homotopy of $\mathbb{P}^1$ in~\cite[Chapter 7.5]{HL}. 
 \end{eg}
 We need the following fact before describing the definably connected sets of $F$. We say $X\subseteq \VF$ is a \textbf{Swiss cheese} if $X=B\backslash (C_1\cup \ldots \cup C_n)$ where $B=\VF$ or a ball and all the $C_i$'s are balls.
 \begin{fact}[{\cite[Proposition 3.33]{Lou}}]\label{fact:C-min}
 Any definable subset $X$ of $F$ is a finite union of disjoint Swiss cheeses.
 \end{fact}
 Using the standard homotopy of $\mathbb{P}^1$, the following is true. It is implicit in~\cite[Chapter 7.5]{HL}.
 \begin{fact}\label{fact:connected}
Let $X\subseteq \VF$ be definable, $X$ is definably connected iff $X$ is a single Swiss cheese.
 \end{fact}
 \begin{Def}\label{def:v+g-cont}
 Let $f:V\rightarrow W$ be a definable function from $V$ to $W$, we say $f$ is \textbf{$v$(respectively $g$)-continuous} if it $f^{-1}(X)$ is $v$(respectively $g$)-open for $X$ a $v$(respectively $g$)-open subset of $W$. We say $f$ is \textbf{$v$+$g$-continuous} if $f$ is both $v$-continuous and $g$-continuous.
 \end{Def}
 \begin{eg}
All polynomial functions $\mathbb{A}^n\to \mathbb{A}^1$ are $v$-continuous. They are also $g$-continuous, which is easily checked from the definition. 
 \end{eg}
 \begin{rmk}
As suggested by the fact that the $v$+$g$-topology is the trace of the topology on $\widehat{V}$ to $V$ in Remark~\ref{rmk:vg}. $v$+$g$-continuous functions are the traces of continuous functions on $\widehat{V}$ in the following sense: Let $f:X\to Y$ be a definable function, the canonical extension $\hat{f}:\widehat{X}\to \widehat{Y}$ is continuous iff $f$ is $v$+$g$-continuous~\cite[Lemma 3.8.2 and 3.8.4]{HL}.
 \end{rmk}

\begin{fact}[{\cite[Theorem 10.4.2]{HL}}]\label{prop:connected}
We have the following:
 \begin{itemize}
     \item If $f$ is $v$+$g$-continuous and $X$ is definably connected, then $f(X)$ is definably connected. 
     \item If $V$ is a geometrically irreducible variety, then $V$ is definably connected.
 \end{itemize}

 \end{fact}
 
 The following is an easy corollary of \cite[Theorem 11.1.1, Proposition 11.7.1]{HL} that suits our purposes.
\begin{fact}\label{comp}
For a definable subset $X$ of some quasi-projective variety in the valued field sort, $X$ has finitely many definably connected components. Furthermore, if $X_t$ is a definable family of subsets, there is $n\in \mathbb{N}$ such that each $X_t$ has at most $n$ definably connected components.
 \end{fact}
 For readers unfamiliar with~\cite{HL}, we briefly explain how this follows from their theorems. Their main theorem states that the stable completion $\widehat{X}$ of $X$ admits a ``deformation retraction" to a definable set in the value group, which has finitely many definably connected components by the o-minimality of $\Gamma$. Thus our statement follows immediately since such ``deformation retraction" preserves definably connected components. The ``furthermore" part follows the same way from the uniform version of their main theorem.

 We need one last definition and theorem for this section. This definition appeared in~\cite[Chapter 4.2]{HL}
 \begin{Def}\label{def:unbounded}
  Let $X\subseteq \VF^n$ be definable, we say that $X$ is \textbf{bounded} if there is $\gamma\in\VG$ such that $v(x_i)\geq \gamma$ for any $x\in X$ and $i=1,\ldots,n$. A type $p\in S_G(F)$ is \textbf{unbounded} if no $X\in p$ is bounded.
 \end{Def}
 The following follows from \cite[Theorem 4.2.20, Proposition 4.2.21]{HL}.
 \begin{fact}\label{fact:bdd}
 Let $X,Y$ be definable sets in the $\VF$-sorts and $f:X\to Y$ be a $v$+$g$-continuous function. If $X\subseteq \VF^n$ is $v$+$g$-closed and bounded, so is $f(X)$.
 \end{fact}
 \begin{prop}
 If $X$ is unbounded and defined over $F$, then there is an unbounded standard type $p\in S_X(F)$.
\end{prop}
\begin{proof}
Take a proper standard extension $F\preceq K$. Since $X$ is unbounded, there is $a \in X(K)$ such that $a\notin X(\conv_F(K))$. Then $tp(a/F) \in S_X(F)$ is standard and unbounded.
\end{proof}
 \section{Proof of the main theorem}\label{sec:proof}
 To compute the $\mu$-stabilizer of a standard type via the standard part map, it is not necessary to work in the monster model. More explicitly, for standard extensions $F\preceq K$, one computes the $\mu$-stabilizers of standard types $p$ realized in $K$ correctly using the standard part, even though $K$ is not saturated. (See Corollary~\ref{cor:def} and Proposition~\ref{prop:st}.)
 \begin{prop}\label{prop:contain}
 Let $p$ be a standard type $p\in S_G(F)$ with a realization $a$ in a standard extension $F\preceq K$. For any $F$-definable set $X$ such that $a\in X$, $\Stab^\mu(p)\subseteq \st(X\cdot a^{-1}\cap G(\conv_F(K)))$.
 \end{prop}
\begin{proof}
We may assume that $X\subseteq G$. Let $g\in \Stab^\mu(p)$, so $g\cdot p\in p_\mu$. Thus $g\cdot p$ is $\mu$-close to $X$. For any $\gamma\in \Gamma(F)$, we have
\[
 p(x)\models \exists y\exists z\,\, y \in G_{>\gamma,e}\wedge z\in X\wedge g\cdot x=y\cdot z.
\]
Working in $K$, we define
\[
I=\{\gamma\in \Gamma:  \exists y\exists z\,\, y \in G_{>\gamma,e}\wedge z\in X\wedge g\cdot a=y\cdot z\}.
\]
Thus $\Gamma(F)\subseteq I$. By o-minimality of $\Gamma(K)$, and the fact that $\Gamma(F)\preceq \Gamma(K)$ as divisible ordered abelian groups, there is $\gamma \in \Gamma(K)$, $\gamma>\Gamma(F)$, $\epsilon\in G_{>\gamma,e}(K)$, and $b \in X(K)$ such that $\epsilon\cdot b=g\cdot a$. Then $g=\st(b\cdot a^{-1})$ by the choice of $\epsilon$.
\end{proof}

\begin{prop}\label{prop:type-iso}
Let $p\in S_G(F)$ be $\mu$-reduced and standard. Let $F\preceq K$ be a standard extension, with $a\in K$ a realization of $p$. Let $V$ be the Zariski closure of $p$. For any $\gamma \in \Gamma(F)$ and any definably connected component $B$ of $G_{>\gamma,e}(K)\cdot a \cap V$, if $b_1,b_2 \in B$, then $tp(b_1/F)=tp(b_2/F)$. 
\end{prop}
\begin{proof}
We first show that any $b \in B$ is $\mu$-reduced. Assume to the contrary, there is a subset $W$ of $G$, with $\dim(W)<\dim(V)$, and $b$ is $\mu$-close to $W$. Since $b\in G_{>\gamma,e}(K)\cdot a\cap V$, $b=g\cdot a$ for some $g\in G_{>\gamma,e}(K)$. The above implies that $a$ is $\mu$-close to $\st(g)^{-1}\cdot W$, which has the same dimension as $W$, a contradiction. In particular, this implies that any $b\in B$ does not concentrate on any proper subvariety of $V$ defined over $F$. In other words, $b$ is a generic point of $V$.

By genericity, no regular function $f$ in $F[V]$ vanishes on $b_i$'s. So it is possible to evaluate all rational functions $f \in F(V)$ on $b_i$'s. Assume to the contrary that the $b_i$'s have different types over $F$.  By quantifier elimination in ACVF, without loss of generality, there must be a rational function $f$ on $V$ such that $v(f(b_1))\leq 0$ and $v(f(b_2))>0$. By  Fact~\ref{fact:C-min},~\ref{fact:connected}, and~\ref{prop:connected}, $f(B)$ is of the form $C\backslash (\cup_i C_i)$, where $C_i$'s are disjoint subballs of $C$.
\begin{claim*}
$f(B)$ contains a $F$-point.
\end{claim*}
\begin{proof}[Proof of claim]
First of all, $C$ is a ball with $a\in C$, $v(a)>0$ and $a'\in C$, $v(a')\leq 0$. Thus $C$ contains the ball $\mathfrak{m}=\{x: v(x)> 0\}$. Note that $a\in \mathfrak{m}$, so $\mathfrak{m}$ is not contained in any of the $C_i$'s. For $f(B)$ to have no $F$-point, the $F$-points of $\mathfrak{m}$ must be covered by the $C_i$'s. Since balls are either disjoint or nested, we may assume that $C_i$'s are all subballs of $\mathfrak{m}$. By the fact that $F\preceq K$ is standard, $\Gamma(F)$ is a convex subgroup of $\Gamma(K)$. If $C_i$'s all have radii greater than $0$,  one can easily produce $a\in \mathfrak{m}$ with $a\notin C_i$ for all $i$. Thus there is $C_i$ with radius $0$, but this is a contradiction to $\mathfrak{m}$ is not contained in any of the $C_i$'s.
\end{proof} 
 Now, we have a generic point $b\in B$ of $V$ over $F$ and $f(b)$ is a $F$-point. Hence $f$ is a constant function, which is a contradiction. Thus established the proposition.
\end{proof}
\begin{lem}\label{thm:bound}
Let $p\in S_G(F)$ be standard, $\mu$-reduced and $a\models p$ with $V$ denoting its Zariski closure. There is $m\in \mathbb{N}$ and $p_1,...,p_m\in S_G(F)$, such that for $g\in G(\conv_F(K))$ with $g\cdot a \in V$, $tp(g\cdot a/F)=p_i$ some $i$.

Moreover, we have a $F$-definable set $X\subseteq G$ with $a\in X$ satisfying the following: For  any $g\in G(\conv_F(K))$,  $g\cdot a \in X$ iff $\st(g)\in \Stab^\mu(p)$ iff $g\cdot a\models p$.
\end{lem}
\begin{proof}
Take $m$ to be the maximum number of definably connected components of the definable family $(G_{>\gamma,e}(K)\cdot a \cap V)_{\gamma \in \Gamma}$ (see Fact~\ref{comp}). If the lemma does not hold for $m$, there exist $g_1,...,g_{m+1} \in G(\conv_F(K))$ such that $g_i\cdot a \in V$ and $tp(g_i\cdot a/F)\neq tp(g_j\cdot a/F)$ for $i\neq j$. Take $\gamma \in \Gamma(F)$ to be such that $g_i \in G_{>\gamma,e}(K)$ for all $i$. This contradicts Proposition~\ref{prop:type-iso}. 

For the ``moreover" part of the lemma, by quantifier elimination, for each $p_i\neq p_j$, there is a rational function $f_{ij}$, $p_i\models v(f_{ij})>0$ and $p_j\models v(f_{ij})\leq 0$. Without loss of generality, $p=p_1$, take $X$ to be defined by the conjunction of $G$ and the valuative inequalities $f_{1j}>0$ on $V$ for $j=2,...,m$. By Proposition~\ref{prop:contain}, $\Stab^\mu(p)\subseteq \st(X\cdot a^{-1}\cap G(\conv_F(K)))$. Conversely, assume that $g\cdot a\in X$ and $g\in G(\conv_F(K))$. By the construction of $X$, we see that $g\cdot a\models p$, so $\st(g)\cdot a \models \mu\cdot p$. Hence $\st(g)\in \Stab^\mu(p)$.
\end{proof}
The ``moreover" part above immediately translates to the following.
\begin{cor}\label{cor:def}
 Let $X$ be given as in Lemma~\ref{thm:bound}, then
 \begin{equation*}
     \st(Xa^{-1})\cap\conv_F(K)^n=\Stab^\mu(p).
 \end{equation*}
\end{cor}
We now proceed to prove Theorem~\ref{thm:main1}.

\begin{thm}\label{thm:infinite}
Let $G$ be a definable (valuatively) closed group in $\VF^n$. Let $p$ be an unbounded standard $G$-type:
\begin{enumerate}
    \item $\Stab^\mu(p)$ is definable.
    \item If $G$ is $g$-closed and the group operations are $v$+$g$-continuous, then $\Stab^\mu(p)$ is unbounded. In particular it is infinite.
\end{enumerate}
\end{thm}
\begin{proof}

For (1), let $X$ be as given in Lemma~\ref{thm:bound} and $a\in K$ be a realization of $p$. By Remark~\ref{rmk:bp}, we have that $F$ is purely stably embedded. Thus by Corollary~\ref{cor:def} and Remark~\ref{rmk:bp}, $\st(X\cdot a^{-1}\cap G(\conv_F(K)))$ is $F$-definable.





We now prove (2). Assume for a contradiction that $\Stab^\mu(p)$ is bounded, then there is some $\gamma\in \Gamma(F)$ such that $\Stab^\mu(p)\subseteq G_{>\gamma,e}(F)$. For a given $a\models p$ and $X$ be as before, we have the following:
\begin{equation*}
    X\cdot a^{-1}\cap G(\conv_F(K))\subseteq G_{>\gamma,e}(K).
\end{equation*}
Thus
\begin{equation*}
    X\cap G(\conv_F(K))\cdot a \subseteq G_{>\gamma,e}(K)\cdot a.
\end{equation*}
Particularly, we have 
\begin{equation*}
    X\cap G(\conv_F(K))\cdot a =X\cap  G_{>\gamma,e}(K)\cdot a.
\end{equation*}
For each $\gamma'\leq \gamma$ and $\gamma'\in \Gamma(F)$,
\begin{equation*}
    X\cap G(\conv_F(K))\cdot a =X\cap  (G_{>\gamma,e}(K)\cdot a)=X\cap (G_{\geq \gamma',e}(K)\cdot a).
\end{equation*}
By $g$-continuity of the group operation, we see that $X\cap (G_{>\gamma,e}(K)\cdot a)=X\cap (G_{\geq \gamma',e}(K)\cdot a)$ is $v$+$g$-clopen in $X$, hence it is a union of definably connected components of $X$. Since $X$ is $F$-definable, all the definably connected components of $X$ are definable over $\mathrm{acl}^{\mathrm{eq}}(F)$. Since $F$ is a model, it follows that the definably connected components of $X$ are $F$-definable. Thus $X\cap (G_{\geq \gamma',e}(K)\cdot a)$ is $F$-definable. However, using Fact~\ref{fact:bdd}, since  $G_{\geq \gamma',e}(K)$ is $v$+$g$-closed, bounded and the group operations are $v$+$g$-continuous, the set $X\cap (G_{\geq \gamma',e}(K)\cdot a)$ is bounded. A contradiction to the fact that $p$ is unbounded.
\end{proof}

Thus we have completed the proof of Theorem~\ref{thm:main1}. 
\subsection{Linear algebraic groups}\label{sec:lin-alg}
We will now look at the special case when $G$ is linear algebraic.

Given a proper standard extension $F\preceq K$, recall that $\conv_F(K)$ is a valuation ring of $K$ and $\Gamma(F)$ is a convex subgroup of $\Gamma(K)$. Let $v$ be the original valuation on $K$. The valuation given by $\conv_F(K)$ can be described via the following:
\begin{equation*}
    v': K\to  \left(\Gamma(K)/\Gamma(F)\right)_\infty,\,x\mapsto v(x)+\Gamma(F).
\end{equation*}
We let $K'$ denote the valued field structure on field $K$ given by the above valuation. From now on, we let $v$ denote the valuation on $K$ and $v'$ denote the valuation on $K'$. For each element $a\in K$, we use $a'$ to denote the same element but viewed as an element in $K'=(K,v')$.

In particular, $K'\models$ ACVF and the standard part map $\st$ can be viewed as the residue map in $K'$. Note that $K'$ can be turned into a model of $T_{loc}$ in \cite[Section 6]{HK3}, where $F$ is interpreted as the embedded residue field, so one can study the theory of $\mu$-stabilizers in the sense of~\cite{muACF}. When $G$ is linear algebraic, on the set level, we see that $G(K)=G(K')$, as the definition of $G$ only relies on the algebraic structure of $F$. Abusing notation, $G$ denotes the linear algebraic group in $K$ and $K'$ at the same time. Similarly, we let $S_K(F)$ to denote the types over $F$ where the ambient structure is $K$ and $S_{K'}(F)$ to denote the types in the $\mathcal{L}_{val}$ over $F$ in $K'$. The next proposition reduces Remark~\ref{rmk:lin-alg} to the main theorem in \cite{muACF} (See Fact~\ref{fact:ksy}).  

Let us recall the setting in~\cite{muACF}. In general, we consider an algebraically closed field $k$ as a trivially valued subfield of an algebraically closed valued field, where we have constant symbols for $k$ both in the $\RES$ and $\VF$-sort. Let $G$ be an linear algebraic group defined over $k$. Consider $G$ as a definable group in the $\VF$-sort and we use $S_G(k)$ to denote the (ACVF) types concentrating on $G$. One can define $\mu$ as the kernel of $G(\mathcal{O})\to G(\RES)$, and it induces a similar equivalence relation as $\mu$ in our setting, and one can develop the theory of $\mu$-stabilizers. We say a type $p\in S_G(k)$ is \textbf{residually algebraic} if there is $K\models\mathrm{ACVF}$, $a\in K$ and $a\models p$ such that $\res(K)=k$. $p$ is \textbf{centered at infinity} if $p$ does not concentrate on $G(\mathcal{O})$.

Next we exhibit the correspondence between unbounded types and types centered at infinity together with the correspondence of standard types and residually algebraic types.
Note that the following result is not a triviality since $K$ may not have enough automorphisms over $F$. 

\begin{prop}\label{prop:res-alg}
Let $a,b$ be tuples in $K$ such that $tp_K(a/F)=tp_K(b/F)=p\in S_K(F)$ is standard and unbounded. Then $tp_{K'}(a'/F)=tp_{K'}(b'/F)=p'\in S_{K'}(F)$, and $p'$ is centered at infinity and residually algebraic. Moreover,
for each $c'\in K$ with $c'\models p'\in S_{K'}(F)$, one has $c\models p  \in S_{K}(F)$.
\end{prop}
\begin{proof}
Let $V\subseteq \mathbb{A}^n$ be the Zariski closure of $p$ over $F$. Clearly the Zariski closure of $p$ over $F$ depends on the algebraic part of $p$ only. So we identify the Zariski closure of $p'$ with $V$. By quantifier elimination, to determine $tp_{K'}(a'/F)=tp_{K'}(b'/F)$, it suffices to determine for each rational function $f\in F(V)$, $v'(f(a'))> 0$ iff $v'(f(b'))> 0$. Let $f$ be a rational function on $V$. By the way $v'$ on $K'$ is defined, it follows that $v'(f(a'))>0$ iff $v(f(a))>\Gamma(F)$ iff $v(f(b))>\Gamma(F)$ iff $v'(f(b'))>0$. Since $\RES(K')=F$, $p'$ is residually algebraic. $p'$ is unbounded follows immediately from the definition. For the ``moreover" part, let $V$ be the Zariski closure of $c'$, to determine $tp_K(c/F)$, one needs to decide that for every $f\in F(V)$, $v(f(c))$ is greater than $0$ or not from the information in $tp_K'(c'/F)$. If $v'(f(c'))>0$, then clearly $v(f(c))>0$. Same for the case $<0$. If $v'(f(c'))=0$, since $p'$ is residually algebraic, there is $b'\in F$ such that $v'(f(c')-b')>0$. Then $v(f(c))=v(b)$. The completes the proof of the proposition.
\end{proof}
From the above we know that $\Stab^\mu(p)=\Stab^\mu(p')(F)$, where the right hand side is computed as in~\cite{muACF}. By Fact~\ref{fact:ksy}, we have the following.
\begin{cor}\label{cor:lin}
 If $p\in S_G(F)$ is $\mu$-reduced, standard and $G$ is linear algebraic, then $\Stab^\mu(p)\subseteq G$ is a solvable algebraic subgroup with $\dim(\Stab^\mu(p))=\dim(p)$.
\end{cor}
\begin{proof}
 If $p$ is unbounded, this follows exactly the statement of ~\cite[Theorem 1.2]{muACF}. When $p$ is bounded and $\mu$-reduced and residually algebraic, $p$ is realized so $\Stab^\mu(p)=\{e\}$ and the statement holds trivially.
\end{proof}
\section{Examples}
In this section, we compute some concrete examples of $\mu$-stabilizers.
 \begin{eg}\label{ex:basic1}
  Let \(G=\mathrm{SL}_{2,F}\). Let
  \[
  X_1=\left\{\begin{pmatrix}
    x & 1 \\
    0 & x^{-1}
  \end{pmatrix} : {x\in F^* }\right\}
  \]
  be a closed subvariety of \(G\). Since this subvariety is isomorphic (as an algebraic variety) to \(\mathbb{G}_m\) via
  \[
  \begin{pmatrix}
    x & 1 \\
    0 & x^{-1}
    \end{pmatrix} \mapsto x,
   \] 
   the definable subset given by \( v(x)< \Gamma(F)\) isolates a complete type \(p\) on \(G\).

  We claim that the $\mu$-stabilizer \(H\) of $p$ is the subgroup
  \[
    \left\{
    \begin{pmatrix}
      1 & a \\
      0 & 1
    \end{pmatrix}: a\in F
    \right\}.
  \]

  Indeed, let \(g\in F\) and choose $\alpha\in K$ with $v(\alpha)<\Gamma(F)$.  
  Then
  \[
  \begin{pmatrix}
    1 & g \\ 0  & 1 
  \end{pmatrix}
  \begin{pmatrix}
    \alpha & 1 \\ 0 & \alpha^{-1} 
  \end{pmatrix}=
  \begin{pmatrix}
    \alpha & 1+g\alpha^{-1} \\
    0 & \alpha^{-1}
  \end{pmatrix}=
  \begin{pmatrix}
    1+\varepsilon & 0 \\
    0 & (1+\varepsilon)^{-1}
  \end{pmatrix}
  \begin{pmatrix}
    \beta & 1 \\ 0 & \beta^{-1}
  \end{pmatrix},\]
  where $\varepsilon =g \alpha^{-1}$ and
  $\beta=(1+\varepsilon)^{-1}\alpha$.
  Since \(v(\varepsilon)>\Gamma(F)\), we have
 \[ \begin{pmatrix}
    1+\varepsilon & 0 \\
    0 & (1+\varepsilon)^{-1}
  \end{pmatrix}\in \mu
  \]and 
 \[
 \begin{pmatrix}
   \beta & 1 \\
   0 & \beta^{-1}
 \end{pmatrix}\models{}p.
\]
So \(H\subseteq \Stab^\mu(p)\). However, $\Stab^\mu(p)$ is $1$-dimensional by~Corollary~\ref{cor:lin}, and a similar computation concludes that $H=\Stab^\mu(p)$. This example is to be compared with~\cite[Example 2.10]{muACF} to illustrate the correspondence between standard types and residually algebraic types via the $\conv_F(K)$ construction (Proposition~\ref{prop:res-alg}), as well as the fact that it preserves the $\mu$-stabilizers.
\end{eg}
\begin{eg}\label{ex:basic2}
On the other hand, it is not necessarily the case that any definable group in the valued field sort is an algebraic group. For example, let $G$ be the subgroup of $\mathrm{GL}_{2,F}$ with the condition that the determinant has valuation $0$. Consider the closed subvariety of $G$ given by
 \[
  X_2=\left\{\begin{pmatrix}
    x & 0 \\
    1 & x^{-1}
  \end{pmatrix} :x \in F^* \right\}.
 \]
Even though the above group $G$ is not algebraic, consider the same type isolated by $v(x)<\Gamma(F)$, a similar computation as~\cite[Example 2.11]{muACF} yields that $\Stab^\mu(p)$ is the following group:
\[
\left\{ \begin{pmatrix}
    x & 0 \\
    0 & x^{-1}
  \end{pmatrix}: x\in F^*\right\}.
\]
In this case, since $p$ has dimension $1$, we trivially have that $\dim(p)=\dim(\Stab^\mu(p))$. Comparing to Example~\ref{ex:basic1}, the fact that $X_2$ have a very different $\mu$-stabilizer from $X_1$ is related to the fact that left and right $\mu$-stabilizers are usually different even for the same type.

In general, higher dimensional $\mu$-stabilizers are hard to compute and we end with the following question.
\end{eg}
\begin{question}
Can Theorem~\ref{thm:infinite} be strengthened to the following: For a $\mu$-reduced standard type $p$, is it true that $\dim(p)=\dim(\Stab^\mu(p))$? More optimistically, is it true that any $\Stab^\mu(p)$ for any standard $p$ is algebraic?
\end{question}

\bibliographystyle{alpha}
\bibliography{bibliography}
\end{document}